\documentclass[12pt]{article}
\usepackage{amsmath,latexsym,amssymb,amsthm,enumerate,amsmath,amscd}
\usepackage{amsthm,amsmath,amsfonts,amssymb,amscd,latexsym,stmaryrd}
\usepackage[vcentermath,enableskew]{youngtab}
\usepackage[mathscr]{eucal}
\usepackage{color}
\usepackage{cite}
\usepackage[colorlinks=true, pdfstartview=FitV, linkcolor=blue, citecolor=blue, urlcolor=blue]{hyperref}
\usepackage[nameinlink]{cleveref}
\usepackage[all,cmtip]{xy}
\usepackage{geometry}

\newcommand{\F}{{\mathbb F}}

\newcommand{\Hom}{\operatorname{Hom}}

\def\m{\mathfrak{m}}
\def\G{\mathrm{G}}

\theoremstyle{plain}
\newtheorem{theorem}{Theorem}[section]
\newtheorem{proposition}[theorem]{Proposition}

\newtheorem{lemma}[theorem]{Lemma}
\theoremstyle{definition}
\newtheorem{definition}[theorem]{Definition}

\newtheorem{remark}[theorem]{Remark}
\newtheorem{example}[theorem]{Example}

\newtheorem*{ack}{Acknowledgment}

\begin{document}
\date{October 19, 2022, revised May 8, 2023}
\author{Anthony A. Iarrobino\\[.05in]
{\small Department of Mathematics, Northeastern University, Boston, MA 02115,
USA.
}}
\title{Log-concave Gorenstein sequences}
\renewcommand\footnotemark{}
\thanks{ \textbf{Email address:} {a.iarrobino@northeastern.edu}}
\maketitle
\abstract  
We show here that codimension three Artinian Gorenstein sequences are log-concave, and that there are codimension four Artinian Gorenstein sequences that are not log-concave. We also show the log-concavity of level sequences in codimension two.
\vskip 0.2cm
{\it  Dedicated to the memory of friend and colleague, Jacques Emsalem.}
\tableofcontents
\section{Introduction.}  A codimension $r$ Gorenstein sequence is here a Hilbert function $H(A)$ that occurs for a codimension $r$ graded Artinian Gorenstein (AG) algebra $A$ over an infinite field $\F$. The codimension two Gorenstein sequences are the same as those for complete intersections $A(a,b)={\F}[x,y]/(x^a,y^b)$ - known to F.H.S.~Macaulay \cite{Mac0,Mac}. The codimension three Gorenstein sequences are known because of the Pfaffian structure theorem of D.~Buchsbaum and D.~Eisenbud \cite{BuEi}, (Lemma~\ref{1lem} below, see also \cite{St,Di}). The Hilbert function $H(A)$ of a graded Artin algebra is a sequence satisfying a certain condition determined by F.H.S.~Macaulay (Equation~\eqref{Mac3eq}, Lemma \ref{Mac2lem} below, \cite{Mac2}, and \cite[\S 4.2]{BrHe}). We will call such a sequence satisfying Equation \eqref{Mac3eq} a \emph{Macaulay sequence}.  The \emph{socle} of an Artinian algebra $A$ is $(0:\m_A)$ where $\m_A$ is its maximal ideal, and the \emph{socle degree} of $A$ is the highest degree of a socle element; when $A$ is Artinian Gorenstein, the socle degree of $A$ is the highest degree $j$ for which $H(A)_j\ne 0$.
We recall 
 \begin{lemma}\label{Mac2lem}[Macaulay's theorem\cite{Mac2}] Let $H=H(A)=(1,r,h_2,\ldots ,h_a,\ldots )$ be the Hilbert function of an algebra quotient of $R=\F[x_1,\,
\ldots,x_r]$. We may write uniquely, for $a\ge 1$, the \emph{Macaulay expansion}
 \begin{equation}\label{Mac2eq}
 h_a=\sum_{i=1}^a\binom{n_i}{i} \text { with } n_a>n_{a-1}>\cdots> n_{1}\ge 0.
 \end{equation}
 Then we have
\begin{equation}\label{Mac3eq}
h_{a+1}\le h_a^{(a)}: = \sum_{i=1}^{a}\binom{n_i+1}{i+1}.
\end{equation}
\end{lemma}
For further discussion see \cite[Lemma 4.2.6,Theorem 4.2.10]{BrHe}.
 When there is equality in Equation \eqref{Mac3eq} we term this \emph{maximum Macaulay growth} from degree $a$ to degree $a+1$ \cite[Section 4.2]{BrHe}.\par

\begin{definition}\label{SIdef} An SI sequence   $H=(1,r,\ldots, r,1_j)$ of socle-degree $j$ is a sequence satisfying both
\begin{align} h_i&=h_{j-i} \text{ for } 0\le i\le j/2;\notag\\
(\Delta H)_{\le j/2}& \text { is a Macaulay sequence},\label{SIeq}
\end{align}
where $(\Delta H)_i=h_i-h_{i-1}$  (take $h_{-1}=0$).
\end{definition}  We have
\begin{lemma}\cite{BuEi,St}\label{1lem}  A sequence $H=(1,3,\ldots, h_i,\ldots ,3,1)$ is a codimension three Gorenstein sequence if and only if $H$ is an SI sequence.
\end{lemma}
The proof follows from the D. Buchsbaum and D. Eisenbud  Pfaffian structure theorem for codimension three Gorenstein algebras \cite{BuEi}; see also \cite{Di} and \cite[Theorem 5.25]{IK}.\par
It is well known that in any codimension, the SI sequences are a subset of the Gorenstein sequences: when $r\ge 5$, the SI sequences are a proper subset of the Gorenstein sequences, which may be non-unimodal; when $r=4$ it is open whether the SI sequences might be all the Gorenstein sequences.  A result of N. Altafi shows that given a finite SI sequence $H$, there is always a strong Lefschetz Artinian Gorenstein algebra of Hilbert function $H$ \cite{Alt}.\par
\begin{definition}\label{lcdef}We say that a finite sequence $H=(h_0,h_1,\ldots, h_i,\ldots,h_j)$ is \emph{log-concave in a degree $i\in [ 1,j-1]$} if 
	\begin{equation}
		\label{eq:hflc}
		h_{i-1}\cdot h_{i+1}\leq h_i^2.
	\end{equation}
	The sequence $H$ is \emph{log-concave} if it is log-concave in each such degree $i$.
	\end{definition}
	See the R. Stanley 1989 survey \cite{St2},  the F. Brenti 1994 \cite{Bre}, and many more recent articles.
	 Log-concavity has a relation with the Hodge-Riemann property of certain algebras \cite{Ba,H,MMS,MuNY}.

\section{Codimension three Gorenstein sequences are log-concave.}
\begin{theorem}
	\label{cod3thm}
	Let $A$ be a standard graded AG algebra of socle degree $j$ with Hilbert function $H(A)=(h_0,h_1,\ldots,h_j)$ satisfying $h_1=3$.  Then the sequence $H(A)$ is log-concave.
	\end{theorem}
\begin{proof}
	Note that it suffices to show that Equation \eqref{eq:hflc} holds for $1\leq i\leq \left\lfloor\frac{j}{2}\right\rfloor$, since for $i>\left\lfloor\frac{j}{2}\right\rfloor$, we have $1\leq j-i\leq \left\lfloor\frac{j}{2}\right\rfloor$ and hence Equation \eqref{eq:hflc} will hold for these $i$ by symmetry of the Hilbert function.  For each $1\leq i\leq \left\lfloor\frac{j}{2}\right\rfloor$, we have $h_i\leq \binom{h_1+i-1}{i}=\binom{i+2}{i}$, and if we have equality for every $i$, then $H(A)$ is log-concave since the binomial coefficients are log-concave.  Otherwise we may choose the smallest index $u, 1\leq u\leq \left\lfloor\frac{j}{2}\right\rfloor$ such that $h_u<\binom{u+2}{u}$.  Then of course Equation \eqref{eq:hflc} holds for $1\leq i\leq u-1$ for the preceding reason, and hence we need only check Equation \eqref{eq:hflc} for $u\leq i\leq \left\lfloor\frac{j}{2}\right\rfloor$. Let $j^\prime==\left\lfloor\frac{j}{2}\right\rfloor$. The following two observations are key:  by Lemma \ref{1lem}
	\begin{enumerate}[(i).]
		\item $H(A)$ is an SI sequence, and hence the first difference 
		\begin{equation}\label{delta1eq}\Delta H(A)_{\le j^\prime}=(1, 2, \ldots,\Delta H_{j^\prime})\end{equation}
		is the Hilbert function for some standard graded Artinian algebra of codimension $2$, and 
		\item The Hilbert function of a standard graded Artinian algebra of codimension $2$ is non-increasing after the initial degree $d$ of its defining ideal (here $d=\min\{i\mid\Delta H(A)_i<i+1\}$) so $ \Delta H(A)_{\le j^\prime}$ in Equation \eqref{delta1eq} satisfies 
		\begin{equation}\label{delta2eq}\Delta H(A)_{j^\prime}=(1,2,\ldots,d,\Delta_d,\ldots, \Delta_{j^\prime}), \text { with } d\ge \Delta_d\ge \Delta_{d+1}\ge \cdots \ge\Delta_{j^\prime}.
		\end{equation}
	\end{enumerate}
	 The second observation is well known (see \cite{Mac0},\cite[Lemma~1.3]{I})\footnote{The article \cite{Mac0} also determines all the Hilbert functions that do occur for local algebras of codimension two, as well as for the graded algebras.} and can be seen as follows:  Let $B=\F[x,y]/I$ be any standard graded Artinian algebra in codimension two, suppose that $I_p\ne 0$ and suppose that $f_1,\ldots,f_m\in I_p$ are linearly independent forms in $I$ of degree $p$.  Then certainly $xf_1,\ldots,xf_m$ are linearly independent in $I$ of degree $p+1$; also, if $f_1$ has maximum $y$-power among the set of $f_i$, then $yf_1,xf_1,\ldots,xf_m\subset \F[x,y]_{p+1}$ are linearly independent, showing that $\dim_\F(I_p)<\dim_\F(I_{p+1})$, and hence $G=H(B)=(1,2,\ldots,g_p,\ldots )$ satisfies
	 \begin{align*}
	 g_p&=\dim_\F(\F[x,y]_p)-\dim_\F(I_p)=(p+1-\dim_\F(I_p)\\\geq& p+2-\dim_\F(I_{p+1})=\dim_\F(\F[x,y]_{p+1})-\dim_\F(I_{p+1})=g_{p+1}.
	 \end{align*}

	Finally, for any integer $i$ satisfying $d\leq i\leq \left\lfloor\frac{j}{2}\right\rfloor$, we must therefore have 
	\begin{align*}
	h_i^2&\geq h_i^2-(h_i-h_{i-1})^2=(h_i-(h_i-h_{i-1}))(h_i+(h_i-h_{i-1}))\\
	&\geq (h_i-(h_i-h_{i-1}))(h_i+(h_{i+1}-h_{i}))=h_{i-1}h_{i+1},
	\end{align*}
	and hence $H(A)$ is log-concave.   
\end{proof}

\section{Codimension four Gorenstein sequences that are not log-concave.}
\noindent
Many codimension four Gorenstein sequences, as $H=(1,4,10,14,10,4,1)$, are log-concave.  We first show that there are codimension four SI sequences that are not log-concave
(Proposition \ref{firstprop}); then we show that there are codimension four SI sequences that are not log-concave for an arbitrarily large consecutive sequence of degrees (Proposition~\ref{multprop}).\par
The codimension four Gorenstein sequences $H$ include the SI sequences, those satisfying $\Delta H_{\le j/2}=(1,3,\ldots )$ is the Hilbert function of a codimension three graded Artin algebra $A=R/I$ (Definition \ref{SIdef}). We first restrict to codimension four SI sequences satisfying  
\begin{equation}\label{deltaeq}\Delta H_{\le j/2}=(1,3,\ldots, r_k, b,c) \text { where } r_k=\binom{k+2}{2}=\dim_\F R_k.
\end{equation}
We denote by $S={\F }[x,y,z,w]$, and let $s_k=\dim_{\F}S_k=(1+3+\cdots +r_k)=\binom{k+3}{3}.$
 Then the sum function of $\Delta (H)_{\le j/2}$ above satisfies $$H_{\le j/2}=\left(1,4,\ldots, s_k, s_k+b,s_k+b+c\right). $$ The log-concavity condition \eqref{eq:hflc} here in degree $k+1$ for $H$ is 
\begin{equation}\label{lceq} s_k(s_k+b+c)<(s_k+b)^2 \text { or, equivalently, } s_k(c-b)<b^2.
\end{equation}
 Keeping $b,c$ constant with $c>b$ then this certainly is negated for $k$ large enough.
 The next idea is to choose $b$ suitably and let $s_k+b$  in degree $k+1$ to $s_k+b+c$ in degree $k+2$ have maximum Macaulay growth (see Lemma \ref{Mac2lem}).  \par
 The dimension $\dim_\F (\F[x,y,z,w]_k)=s_k$. For codimension three, if $h_a<r_a$ we will denote by $\delta (h_a) =h_a^{(a)}-h_a$: here, this is just the number of terms in the Macaulay expansion of  Equation \eqref{Mac2eq} with $n_i=i+1$. Then, taking  $a=k+1,h_a=s_k+b, s_{k+2}=h_a+c$ with $c=h_a+\delta(h_a)$ (maximum Macaulay growth) the log-concavity condition Equation \eqref{lceq} becomes $\delta\cdot s_k\le b^2$. Thus, to violate log-concavity in degree $k+1$ for a Hilbert function sequence $H$ having as its key entries $h_k=s_k, h_{k+1}=s_k+b, h_{k+2}=s_k+b+\delta$ we need only assure
 \begin{equation}\label{nlceq}\delta\cdot s_k>b^2.
 \end{equation}
 \begin{remark}\label{boundrem}
 Recall that the Gotzmann regularity degree of the constant polynomial $\{s\}$ is itself $s$ (see \cite[Proposition C.32]{IK}). This implies for an SI sequence $H$ that once $(\Delta H)_i\le s$ for an integer $i\in [s,j/2]$ then $(\Delta H)_{\le j/2}$ is non-increasing in higher degrees than $i$.  Also, in order for $\delta=h_a^{(a)}-h_a\ge 2$ we must have $h_a\ge 2a+1$, with equality when $h_a= \binom{a+1}{a}+\binom{a}{a-1}$. For $\delta=h_a^{(a)}-h_a\ge 3$ we need $h_a\ge 3a$, with equality when  $h_a= \binom{a+1}{a}+\binom{a}{a-1}+\binom{a-1}{a-2}$. Evidently, for $\delta\ge 4$ we need 
 \begin{align}
 b=h_a&\ge \delta\cdot a-(2+\cdots+(\delta -2))\notag\\
 &=\delta \cdot a-\delta (\delta-3)/2.\label{keq}
 \end{align}
 These inequalities for $\delta\ge 2$ will greatly affect our search for small examples of SI  sequences in codimension four that are not log-concave - that satisfy Equation~\eqref{nlceq}.
 \end{remark} 
 
We will denote by $H_{a\to b}$ the subsequence $(h_a,h_{a+1},\ldots,h_b)$ of $H$.
 
\begin{proposition}[SI sequences in 4 variables that are not log-concave]\label{firstprop}
We give a series of minimal examples, depending on the choice of  $\delta$. \vskip 0.2cm\noindent
{\bf Case $\delta=1$.} First we consider $\delta=1$ and take $b=\binom{s+1}{s}$, in degree  $s=k+1$. We need $s_k>b^2=(s+1)^2$. Taking $\delta=1,b=7_6, k=5$, so $s_5=56>7^2$ we have $$H=(1,4,10,20,35,56,63,71,63,56,35,20,10,4,1_{14}),\quad  n=449,$$ then $h_5\cdot h_7=56\cdot 71=3976>3969=63^2=h_6^2$.\par
Taking $\delta=1,b=8_7,k=6$ so $s_6=84>8^2$ we have
$$H=(1,4,10,20,35,56,84,92,101,92,84,56,35,20,10,4,1_{16}),\quad n=705, $$
 then $h_6\cdot h_8=84\cdot 101=8484>8464=92^2=h_7^2$. Since $84>9^2$ we have a second example  where $\delta=1, b=9$, $H_{8\to12}=(84,93,103,93,84)$ of socle degree $16$ and length $n=709$, also not log-concave in degree $7$, as $h_6\cdot h_8=84\cdot 103=8652>8649=93^2=h_7^2.$
 \par
 In general, taking $\delta=1, k>> 5$, we may choose $b_{k+1}$ (that is, $b$ in degree $k+1$) satisfying $k+2\le b\le (k/6)^{3/2}$  (asymptotically, not for small $b$) that will satisfy the conditions of Remark \ref{boundrem} and also satisfy $s_k>b^2$, Equation \ref{nlceq}, so we will obtain again a non log-concave $H$.\vskip 0.2cm\noindent
 {\bf Case $\delta=2$}.
Now taking $\delta=2, b=23$ we have $2s_{10}=2(286)=572>529=23^2$ and $23$ satisfies 
$23=2(11)+1$, the lower bound from Remark \ref{boundrem}.  This is the lowest pair $\delta=2,b=23$ with $a=11$ 
satisfying Equation \eqref{nlceq}, and leads to an example of non log-concave
$H$ of socle degree $24$, whose key entries are
\begin{equation}\label{lowestdelta2eq}
H_{10\to14}=(286,309,334,309,286),
\end{equation}
satisfying $h_{10}\cdot h_{12}=286\cdot 334=95524>95481=309^2=h_{11}^2$, so $H$ of length $n=2954$ that is non  log-concave in degree $11$.
\par
Now taking $\delta = 2, b=25$ we have $2s_{11}=2(364)=728>625=25^2=\Delta h_{12}^2$ so we have  new  $H$ of socle degree $26$ whose key entries are
$$H_{11\to15}=(364,389,416,389,364),$$
satisfying $h_{11}\cdot h_{13}=364\cdot 416=151424>151321=389^2=h_{12}^2$, so $H$ of length $n=4034$ that is non  log-concave in degree $12$. Evidently, we may replace $25$ by $26$ as
also $728>26^2=676$. Then the key entries of $H$ would be
$$\text { for } b=26, H_{11\to15}=(364,390,418,390,364).$$
This sequence of examples with $\delta=2$ can evidently be continued, the next arises from $\delta=2, b=27$, $2s_{12}=2(455)=910>729=(27)^2$, so gives an SI sequence $H$ of socle degree $28$ with
key entries
$$H_{12\to16}=(455,482,511,482,455),$$
satisfying $h_{12}\cdot h_{14}= 232505>232324=482^2=h_{13}^2$. Evidently, we may replace $27$ by $b\in [27,30]$ as $910>30^2$.\vskip 0.2cm
The general $\delta=2$ case with fixed $k$ and lowest $b=2 k+3$ will be $s_k,  k\ge  10$ satisfying $2s_k>(2k+3)^2$, leading to an SI sequence $H=(1,4,\ldots,4,1_{2k+4})$ of socle degree $2k+4$, with key entries
$$H_{k\to (k+4)}=(s_k,s_k+2k+3,s_k+4k+8,s_k+2k+3,s_k),$$
of length $n=2\binom{k+4}{4}+3s_k+8k+14,$ that is not log-concave in degree $k+1$. \par Given $k$, the maximum $b$ satisfying Equation \eqref{nlceq} is $\sqrt{2s_k}$; for  $k=25$ this is $b=\lfloor\sqrt{2s_{25}}\rfloor=\lfloor\sqrt{6552}\rfloor=80$.  When $(k,b)$ is a fixed pair, satisfying $2k+3\le b\le \sqrt{2s_k}$ the key entries in the corresponding non log-concave $H$ of socle degree $2k+4$ are
$$H_{k\to (k+4)}=(s_k,s_k+b,s_k+2b+2,s_k+b,s_k),$$
and $H$ has length $2\binom{k+4}{4}+3s_k+4b+2$.
\vskip 0.2cm\noindent
{\bf Case $\delta=3$}. We have $3\cdot s_{18}=3\cdot \binom{20}{3}=3420>3249=9(19)^2$, so taking $(\delta,k,b)=(3,18,3\cdot 19)$ we generate a smallest $H$ for $\delta=3$, here of socle degree $j=40$ that is not log-concave in degree $19$, that has key $H$ values
$$H_{18\to 22}=(1140,1197,1257,1197,1140).$$
\vskip 0.2cm\noindent
{\bf Case $\delta\ge 4$.}  To satisfy Remark \ref{boundrem} and Equation \eqref{nlceq}, we must have $\delta\cdot  s_k\ge b^2$, or $\delta \cdot\binom{k+3}{k}> b^2$, where also $b\ge \delta k$.
 Solving for $b_{\max}$, and approximating $\binom{k+3}{k}$ by $k^3/6$, we have $\delta k^3/6\le b_{\max}^2$, so if $b$ is in the interval
 \begin{equation}\label{beq}  b\in [\delta \cdot k, \,\delta^{1/2}\cdot k^{3/2}/6^{3/2}],
 \end{equation}
 then the triple $(\delta,k,b)$ will produce a codimension four SI sequence $H$ with key entries 
 $$H_{k\to k+4}=(s_k,s_k+\delta\cdot b, s_k+2\delta\cdot b, s_k+\delta\cdot b,s_k)$$ 
 and socle degree $2k+4$ that is non log-concave in degree $k+1$.
\end{proposition}
\begin{example}[Lengthening the non log-concave examples $H$]\label{lengthenex}
The lowest $\delta=2$ example is Equation \eqref{lowestdelta2eq} of socle degree $j=24$, where $H_{10\to14}=(286,309,334,309,286)$ and $\Delta H_{10\to12}=(r_{10},r_{10}+23,r_{10}+25)=(66,89,91)$. We may lengthen this by adding a further sequence after $\Delta H_{12}=114$ that satisfies the Macaulay inequalities Equation \eqref{Mac3eq} for codimension three. For example we may adjoin to $\Delta H$ the sequence $(16_{13},17_{14},8_{15},7_{16})$, leading to a key sequence
$$H^\prime_{10\to22}=(286,309,334,350,367,375,382,375,367,350,334,309,286),$$
of an AG Hilbert function $H^\prime $ of socle degree $j=32$ that is still non log-concave in degree $11$.
\end{example}
In the next Proposition, we extend the maximum growth portion of $\Delta H$, by a distance $\ell$, obtaining in some cases, especially when $s_b>>b^2$ a new AG sequence $H^\prime(\ell)$  having a longer consecutive subset of non log-concave adjacent triples, whose length we can specify. Further, we specify $h_{k+u}h_{k+u+2}-h_{k+u+1}^2$ for each $u\in [0, \ell]$, showing that it is a sequence with linear first differences.
\begin{proposition}[Sequences $H$ having multiple non log-concave places]\label{multprop}
Assume that $(\delta,k,b)$ is a triple as in Proposition \ref{firstprop} for which $\delta\cdot s_k>b^2$, and let $H$ be  the related codimension four Hilbert function of socle degree $j=2k+4$, whose key part is
$$H_{k\to k+4}=(s_k,s_k+b,s_k+2b+\delta, s_k+b,s_k)$$
 and which is non log-concave in degree $k+1$. Let $\ell$ be a positive integer.
We define an extended sequence $H^\prime(\ell)$ identical to $H$ in degrees $i\in [0,k+2]$, and satisfying
 \begin{align}
 \Delta H^\prime(\ell)_{t}&=\Delta T_t \text { for } t\le k+1 \text { and for } t> k+\ell+2;\notag\\
 \Delta H^\prime(\ell)_{k+1\to k+\ell+2}&=(b,b+\delta, b+2\delta,\ldots, b+(\ell-1)\delta,b+(\ell)\delta, b+(\ell+1)\delta).\label{deltaprimeeq}
 \end{align}
 Then $H^\prime(\ell)$ has socle degree $2k+4+2\ell$, and we have for each $u\in [1,\ell+2]$
 that  $h^\prime_{k+u}=s_k+bu+\delta (1+2+\cdots +(u-1))$. Letting $\delta=1$ we have for each such $u$
 \begin{equation}\label{exteq}
 {h^\prime}_{k+u}\cdot {h^\prime}_{k+u+2}-({h^\prime}_{k+u+1})^2= s_k-b^2-((b+1)+\cdots +(b+u)).
 \end{equation}
 Assume now that 
 \begin{equation}\label{upperbound3eq}
 b\ell+\frac{\ell(\ell+1)}{2} < s_k-b^2.
 \end{equation}
  Then $H^\prime(\ell)$ is non log-concave in the positions $k+1,k+2,\ldots, k+\ell, k+\ell+1$.\par
\end{proposition}
\begin{proof} We generalize the proof of Equations \eqref{lceq} and \eqref{nlceq}. Beginning with $(\delta=1, b, k)$  then using the definition of $\Delta H^\prime$ in Equation \eqref{deltaprimeeq} and recalling that $H_k=s_k$ we conclude that $H^\prime=(1,h^\prime_1,\ldots)$ satisfies 
$$\text { for } u\in [0,\ell],\,\, h^\prime_{k+u}=s_k+bu+\sum_{\nu=1}^{u-1}\nu.$$
Then, similarly to Equation \eqref{lceq} we have for $u\in [1,\ell+1]$ and $ i=k+u$
\begin{align}  {h^\prime}_{i-1}{h^\prime}_{i+1}-{h^\prime}_i^2&=({h^\prime}_i-(b+u-1)))({{h^\prime}}_i+b+u)-{h^\prime_i}^2\notag\\
&={h^\prime}_i-(b+u-1)(b+u)\notag\\
&=s_k-b^2-b(u-1)-\sum_{\nu = 1}^{u-1}\nu\notag\\
&=s_k-b^2-b(u-1)-\frac{u^2-u}{2}.
\end{align}
The last statement of the Proposition concerning non log-concavity for $H^\prime(\ell) $ for 
$\ell$ satisfying Equation \eqref{upperbound3eq} follows.
\end{proof}
\begin{example}\label{repeatex}
We take  from Proposition \ref{firstprop} the first sequence $H$ where $\delta=1, k=8, b=10_9$ with key entries $H_{8\to 12}=(165_8,175,186,175,165)$ and recall that $h_8h_{10}-h_9^2=30690-30625=65=s_8-b^2=165-10^2.$ We choose $\ell$ as in Equation \eqref{exteq}, the maximum such that for $b=10$
$$(b+1)+(b+2)+\cdots +(b+\ell)< 65,  \text { from  which we have } \ell=4.$$ We now
consider 
$\Delta H^\prime(4)_{\le j/2}=(1,3,6,10,15,21,28,36,45,10_{9},11,12,13,14,15),$ so
\begin{align}H^\prime(4)&=(1,4,10,20,35,56,84,120,165,175_9,186,198,211,225,240_{14},\notag\\
&\quad 225, 211,198,186, 175, 165,120,84,56 ,35,20,10,4,1_{28}).
\end{align}
of socle degree $28$, which is non log-concave in degrees $9,10,11,12,13$ with
successive  differences (writing $h_i$ for ${t^\prime}_i$)
$$h_{i-1}h_{i+1}-h_i^2=(65_9,54_{10},42_{11},29_{12},15_{13}),$$ whose second differences are $(11,12,13,14)$, consistent with Equations \eqref{deltaprimeeq} and \eqref{exteq}. Note that $H^\prime(5)_{13,14,15}=
(225,240_{14},256)$ whence $h_{i-1}h_{i+1}-h_i^2=0$ for $i=16$, so $H^\prime(5)$ begins log-concavity in degree $16$; this is again consistent with Equation \eqref{exteq} and shows that the maximum length of a non log-concave sequence for  $H^\prime$ arising from this $H$ is five.
\end{example}

Note, that combining with the idea of Example \ref{lengthenex}, we can make codimension four Gorenstein sequences $H^\prime$ with certain specified consecutive subsequences of degrees where
$H^\prime$ is non log-concave, but that may have gaps between them of degrees where $H^\prime$ is log-concave.

\begin{remark}
According to \cite[Theorem 3]{H} these non log-concave $h$-vectors cannot be the $h$-vectors of  a matroid representable in characteristic zero.  Chris McDaniel poses the question of whether a related $f$-vector: $f(t)=h(1+t)$ of one of these non log-concave $h$ might be non-unimodal.
\end{remark}
\section{Higher codimension Gorenstein sequences, level sequences.}

In codimension at least five, Gorenstein sequences need not be unimodal. For example, R. Stanley used an idealization construction \cite{Na} or \cite[\S 3.1]{IMM} to show $H=(1,13,12,13,1)$ is a Gorenstein sequence for an algebra $A\times  \Hom(A,{\F})$ where $A= {\F}[x,y,z]/(x,y,z)^4$ of Hilbert function $(1,3,6,10)$ \cite{St},  Analogous examples, sometimes requiring higher socle degree $j$, can be made in codimensions at least five: in codimension five the lowest socle-degree non-unimodal Gorenstein sequence known has $j=16$ and contains the subsequence
$(h_7,h_8,h_9)=(91,90,91)$ \cite{BeI} (see also \cite{MZ,MNZ} for further discussion of unimodality for Gorenstein sequences). 
\subsection{Higher codimension SI sequences.} SI sequences by construction are unimodal. We note that, as in codimension four, some SI sequences in codimension five or larger are log-concave, and some are not. We illustrate both with the next example.
\begin{example}[Codimenson five] (i). Consider the codimension five SI sequence $H(j)=(1,5,t_2,\ldots)$ having socle degree $j$ with $$\Delta H(j)=(1,4,10,14,20, 27,35,44,\ldots )$$  where $\Delta H(j)_k=\binom{k+2}{k}+\binom{k}{k-1}+\binom{k-2}{k-2}$ for $3\le k\le j/2$ and $j\ge 6$. Then $H(j)_{\le j/2}=(1,5,15,29,49,76,111,155,\ldots,)$ where
\small $$H(j)_k=-1+\binom{k+3}{k}+\binom{k+1}{k-1}+\binom{k-1}{k-2} \text{ for } 3\le k\le j/2,$$
\normalsize
 and  $H(j)_k=H(j)_{j-k} \text { for } k\ge j/2.$
 It is easy to verify that for $j\ge 6$ each such $H(j)$ is log-concave. The reason is that  the sequence $H(j)_k$ for $k\le j/2$ is dominated by its fast-growing term $\binom{k+3}{k}$ which is cubic in $k$.\vskip 0.2cm
 (ii). This suggests that the way to create non log-concave examples of SI sequences is to imitate the construction in codimension four, that is, set $\Delta H^\prime = (1,4,r_2,\ldots, r_{u-1},\delta_u,\ldots)$ where the non-zero terms of the Macaulay expansion of $\delta_u$ have the form $\binom{w+1}{w}$  - that is, we assume that after a certain degree $u$ the  first difference $\Delta H^\prime$ has maximal Macaulay growth, and is linear in $k$.  Consider then
  $\Delta H^\prime=(1,4,10,20,5,6,7,\ldots)$ so that $H^\prime=(1,5,15,35,40,46,53, 61,
 70,\ldots$. Here $35\cdot 46=1610>40^2$, so taking $j=10$ and $H^\prime=(1,5,15,35,40,46,40,35,15,5,1_{10})$ gives a non-log-concave sequence. Likewise, $40\cdot 53>46^2$, so taking
 $j=12$  the SI sequence $H^\prime=(1,5,15,35,40,46,53,46,40,35,15,5,1_{12})$ is non-log-concave in two adjacent places. But it stops there, as $46\cdot 61=2806<2809=53^2$.\par
To get a longer non-log-concave consecutive subsequence, it suffices to begin the minimal linear growth later.
For example, taking $\Delta H^\prime= (1,4,10,20,35,6,7,8,\ldots)$ so 
\begin{equation}  H^\prime_{\le j/2}=(1,5,15,35,70,76,83,91,100, 110,121\ldots),
\end{equation} we have $h^\prime_{k-1}\cdot h^\prime_{k+1}>{h^\prime_k}^2$ for $k=4,5,6,7,8$ with equality for $k=9$, so taking
$$H^\prime=(1,5,15,35,70,76,83,91,100, 110,100,91,83,76,70,35,15,5,1_{18})$$ gives 5 adjacent non-log-concave entries centered in degrees $4$ to $8$. Taking $j=20$ will give
equality in the log-concavity equation in degree 9. \par
\end{example}
\subsection{Log-concavity of level sequences.} Recall that the \emph{socle} of an Artinian algebra $A$ is $(0:\m_A)$ where $\m_A$ is its maximal ideal. A \emph{level sequence} is a Hilbert function possible for an Artinian algebra $A=R/I, R={\F}[x_1,\ldots,x_r]$ having socle all in the same degree, that is, $H(A)=(1,r,\ldots,t_j,0)$ and the socle of $A=(0:\m_A)=  A_j$: we say $A$ has \emph{type} $t_j$.
The Hilbert functions of level algebras are well known in codimension 2, essentially due to F.H.S.~Macaulay \cite{Mac0} but they are not known for type $t$ in higher codimension except for the Gorenstein codimension three case
$(r,t)=(3,1)$ (for some partial results see \cite{ChoI,Za1}). Given a Hilbert function sequence $H=(\ldots ,h_i,\ldots), $ we denote by $e_i=h_{i-1}-h_i=-\Delta(H)_i$. \par
We have \cite[Proposition 2.6, Lemma~2.15]{I3}
\begin{lemma} The Hilbert function $H=H(A)$ of a level quotient $A=R/I$ of $R={\F}[x,y]$ satisfies
\begin{equation}\label{2leveleq}
H=(1,2,\ldots, d,t_d,t_{d+1},\ldots,t_j,0)\text { where } t_j=e_{j+1}\ge e_j\ge e_{j-1}\ge \cdots d-t_d.
\end{equation}
\end{lemma}
A \emph{compressed} level algebra of codimension $r$, socle degree $j$ and type $t$ is one having Hilbert function
\begin{equation}\label{compeq}
H(A)_i=\min\{r_i, t r_{j-i}\},
\end{equation}
where $r_i=\dim_{\F} R_i$ and $R={\F}[x_1,\ldots,x_r]$ (see  \cite{I2} and \cite[Definition 1.3]{Bo}).
For example, with $(r,t,j)=(3,2,5) $ the compressed level Hilbert function $H=(1,3,6,10,6,2)$.
\begin{proposition} (i). Every codimension two level sequence is log-concave.  (ii). Every compressed level Hilbert function in any codimension is log-concave.
\end{proposition}
\begin{proof} The proof of (i) is essentially a translation of the proof of Theorem~\ref{cod3thm}, with Equation \eqref{2leveleq} replacing Equation \eqref{delta2eq}, there.  The proof of (ii) follows from the log-concavity of the sequence $(1,r_1,r_2,\ldots, r_i\ldots)$ where $r_i=\dim_{\sf k} R_i=\binom{r+i-1}{r}$ - as, letting $e=e(r,t,j)$ be the highest degree such that $H(A)_i=r_i$ in Equation \eqref{compeq}, for $i<e$ we have
$h_i^2=r_i^2\ge r_{i-1}\cdot r_{i+1}=h_{i-1}\cdot h_{i+1}$; and $h_e^2=r_e^2\ge r_{e-1}r_{e+1}\ge h_{e-1}h_{e+1}$; and for $i>e+1$ we have $h_i=tr_{j-i}$, so likewise $h_i^2\ge h_{i-1}h_i$; for $i=e+1$ we have, since $ tr_{j-e}\ge h_e,$
 $$h_i^2=t^2r^2_{j-(e+1)}\ge tr_{j-e}\cdot tr_{j-(e+2)}\ge h_e\cdot h_{i+1}.$$
\end{proof}
However, there are codimension three level sequences that are not unimodal - see 
\cite{We,AhS}, the former giving a non-unimodal example with $r=3,t=5$.\par
\subsubsection{Problems}
A \emph{pure} O-sequence is a Hilbert function $H=H(A)$ that can occur for a monomial level algebra $A=R/I$ of socle degree $j$: that is, $I_j$ has a basis of monomials. In codimension three, a type two monomial Artinian algebra is known to be weak Lefschetz (there is a linear form $\ell$ such that $m_\ell: A_i\to A_{i+1}$ has maximum rank for each $i$), provided $\F$ has characteristic zero \cite[Theorem 6.2]{Bo-Z}.  This implies that $H(A)$ is differentiable (the first difference is an O-sequence) until its maximum value, which may be repeated, then $H(A)$ is decreasing \cite[Theorem~ 12]{MNZ2}).  Although their Hilbert functions (codimension three, type two) are characterized in \cite[Proposition 6.1]{Bo-Z}, it remains to see if these sequences are log-concave.  \par
There is a subtle connection of pure O-sequences with algebras associated to matroids
(see \cite{CM,C-T,H,MuNY}).\par
Recall that a sequence $H=(1,2,\ldots,d-1,d,h_d,h_{d+1},\cdots,h_j)$ is \emph{admissible of decreasing type} if $d\ge h_d\ge \cdots$) and there is
$b\ge d$ such that 
\begin{equation}
d=h_{d-1}=\cdots =h_{b-1}>h_b>h_{b+1}>\cdots>h_j.
\end{equation}
Recall that an $\sf h$-vector of a Gorenstein domain $D$ of dimension $s$ is the Hilbert function $H(A), A=D/(\ell_1,\ldots, \ell_s)$ where $\ell_1,\ldots ,\ell_s$ is a regular sequence of linear forms. The $\sf h$-vectors of Gorenstein domains are a proper subclass of Artinian Gorenstein sequences.  When $H=(1,2,\ldots, d,h_d,h_{d+1},\dots ,h_j )$ the condition for $H$ to be an $\sf h$-vector is for $H$ to be admissible of decreasing type \cite{GP}. The AG sequence $H=(1,3,\ldots )$ as in Lemma \ref{1lem} is an $\sf h$-vector of a Gorenstein domain if and only if the difference sequence of Equation \eqref{delta2eq} is admissible of decreasing type \cite{DV},\cite[Theorem 2.18]{Va}.  A characterization of the Hilbert function of Gorenstein domains in higher codimension $h_1\ge 4$ is not known. Also, it is apparently still open whether such $\sf h$-vectors  $(1,h_1,\ldots)$ of domains with $h_1\ge 4$ would be log-concave \cite[Conjecture 5.2]{Bre}. \par While the deformation properties of Artinian algebras have been widely studied (see, for example, \cite{EmI,Di,BoI,CJN,AEI}) it appears open to check whether there are special properties, pertaining, say to smoothability - deformations to a smooth scheme - depending on the log-concavity of the Hilbert functions $H$. What can we say about the irreducible components of the variety $\G_T$ parametrizing graded Artinian algebras of Hilbert function $H$ when $r\ge 4$ and $H$ is log-concave?\par
Subsequent to our work, and partly inspired by it, F. Zanello has further studied which level sequences and pure O-sequences are log-concave, resolving the problem in some cases and proposing additional problems \cite{Za2}.

\begin{ack} We appreciate discussions with Chris McDaniel, who proposed the problem of determining which Artinian Gorenstein sequences might be log-concave, and suggested the reference \cite{H}.  Juan Migliore suggested a mention of level sequences of codimension three, which need not be log-concave, and asked about log-concavity for pure level $O$ sequences of low type in codimension three. We appreciate comments of F. Zanello. We are thankful for helpful comments of the referee.\par
I am grateful for the long friendship and collaboration with Jacques Emsalem, which began in 1975-1976 during an NSF-CNRS sponsored year in Paris - I was visiting the group of Monique Lejeune-Jalabert, Bernard Teissier, and L\^{e} Dung Trang. They put Jacques and I in contact, as we were working on the same problem of parametrizing nets of conics (see the translation and revision with Nancy Abdallah \cite{AEI}). We were very recently collaborating on a sequel project when, after an illness, Jacques passed on July 12, 2022.
This note is very much in the hands-on style of our work together.
\end{ack}


\small
   \begin{thebibliography}{cccccccc}
   
    \bibitem[AEI]{AEI} N. Abdallah, J. Emsalem and A. Iarrobino: \emph{Nets of Conics and associated Artinian
algebras of length 7, Translation and update of the 1977 version by J. Emsalem and A. Iarrobino}, European Journal of Mathematics 9, 22 (2023). https://doi.org/10.1007/s40879-023-00600-9. (arXiv:math.AG/2110.04436).
 
 \bibitem[AhS]{AhS}
 J. Ahn and Y.S. Shin: \emph{Artinian level algebras of codimension 3},  J. of Pure and Applied Algebra, 216 (2012), 95-107.
 
   \bibitem[Alt]{Alt}
N. Altafi, \emph{Hilbert functions of Artinian Gorenstein algebras with the strong Lefschetz property,} Proc. Amer. Math. Soc. 150 (2022), no. 2, 499--513. 

\bibitem[Ba]{Ba}
M. Baker: \emph{Hodge theory in combinatorics,}
Bull. Amer. Math. Soc. (N.S.) 55 (2018), no. 1, 57-80. 

\bibitem[BeI]{BeI}
D. Bernstein and A. Iarrobino: {\it A nonunimodal graded Gorenstein Artin
 algebra in codimension   
    five}, Comm. in Algebra 20(8),
    (1992), 2323--2336. 

\bibitem[Bo]{Bo}
M. Boij: \emph{Betti numbers of compressed level algebras,} J. Pure Appl. Algebra 134 (1999), no. 2, 111-131. 

\bibitem[BoI]{BoI}
M. Boij and A. Iarrobino: \emph{Reducible family of height three level algebras},
Journal of Algebra 321 (2009), 86--104.

\bibitem[BMMNZ]{Bo-Z}
M. Boij, J. Migliore, R. M Mir\'{o}-Roig, U. Nagel,  F. Zanello: \emph{On the shape of a pure O-sequence,} Mem. Amer. Math. Soc. 218 (2012), no. 1024, viii+78 pp. ISBN: 978-0-8218-6910-9.

\bibitem[Bre]{Bre}
F. Brenti: \emph{Log-concave and unimodal sequences in algebra, combinatorics, and geometry: an update}. Jerusalem combinatorics '93, 71-89, 
Contemp. Math., 178, Amer. Math. Soc., Providence, RI, 1994. 

  \bibitem[BrHe]{BrHe}
 W. Bruns and J. Herzog: 
\emph{Cohen-Macaulay rings},
Cambridge Studies in Advanced Mathematics, 39. Cambridge University Press, Cambridge, (1993. xii+403 pp. ISBN: 0-521-41068-1.

\bibitem[BuEi]{BuEi}
D. Buchsbaum and D. Eisenbud: \emph{Algebra structures for finite free
resolutions, and some structure theorems for codimension three},
Amer. J. Math. \textbf{99} (1977), 447--485.

\bibitem[CJN]{CJN}
G. Casnati, J. Jelisiejew, and R. Notari: \emph{Irreducibility of the Gorenstein loci of Hilbert schemes via ray families,}
Algebra Number Theory 9 (2015), no. 7, 1525-1570. 

\bibitem[ChoI]{ChoI}
Young Hyun Cho and A. Iarrobino:  \emph{Hilbert functions and level algebras}, J. Algebra 241 (2001),
745--758.

\bibitem[CM]{CM}
A. Constantinescu, M. Mateev: \emph{Determinantal schemes and pure O-sequences}
J. Pure Appl. Algebra 219 (2015), no. 9, 3873-3888.

\bibitem[C-T]{C-T}
P. Cranford, A. Dochtermann, E. Haithcock, J. Marsh, S. Oh,  and A.~Truman: \emph{Biconed graphs, weighted forests, and h-vectors of matroid complexes} Electron. J. Combin. 28 (2021), no. 4, Paper No. 4.31, 23~pp.

\bibitem[DV]{DV}
E. DeNegri and G. Valla: \emph{The $h$-vector of a Gorenstein codimension three domain,}
Nagoya Math. J. 138 (1995), 113-140.

   \bibitem[Di]{Di}
S. J. Diesel: \emph{Some irreducibility and dimension theorems for
families of height 3 Gorenstein algebras}, Pacific J. Math. \textbf{172} (1996),
365--397.

\bibitem[EmI]{EmI}
J. Emsalem and A. Iarrobino, \emph{Some zero-dimensional generic singularities: finite algebras 
 having small tangent spaces}, Compositio 
    Math. 36 (1978), 145--188.
    
   \bibitem[GP]{GP}
    L. Gruson and C. Peskine: \emph{Genre des courbes de l'espace projectif,} Algebraic geometry (Proc. Sympos., Univ. Troms\o, 1977), pp. 31-59, 
Lecture Notes in Math., 687, Springer, Berlin, 1978.

\bibitem[H]{H}
J. Huh: \emph{$h$-vectors of matroids and logarithmic concavity}, Advances in Math. 270 (2015), 49-59.

\bibitem[I1]{I}
A. Iarrobino: \emph{Punctual Hilbert schemes}, 111 p., Mem. Amer. Math. Soc., Vol. 10 (1977),    \#188, 
   American  Mathematical Society, Providence, (1977).
   
   \bibitem[I2]{I2}
   A. Iarrobino: \emph{Compressed algebras: Artin algebras having given socle 
    degrees and maximal length}, Trans. AMS 285 (1984),337--378
   
   \bibitem[I3]{I3}
   A. Iarrobino: \emph{Ancestor ideals of vector spaces of forms, and level
algebras}, J. Algebra {\bf 272} (2004), 530--580. 
   
   \bibitem[IK]{IK}
  A. Iarrobino and V. Kanev, {\it{Power Sums, Gorenstein Algebras, and Determinantal
 Varieties}}, Appendix C by Iarrobino and Steven L. Kleiman {\it The Gotzmann Theorems and the Hilbert scheme}, Lecture Notes in Mathematics, 1721. Springer-Verlag, Berlin, (1999).  
 
 \bibitem[IMM]{IMM}
A. Iarrobino, P. Macias~Marques, and C. McDaniel: \emph{Artinian
    algebras and Jordan type}, arXiv:math.AC/1802.07383 v.5 (2020), 42p., To appear
J. Commut. Algebra.,  Vol. 14, \#3 (2022).

 
  \bibitem[Mac1]{Mac0}
  F.H.S.~Macaulay: \emph{On a method of dealing with the intersections of plane curves}, 
Trans. Amer. Math. Soc. 5 (1904), no. 4, 385-410.

\bibitem[Mac2]{Mac}
F.H.S.~Macaulay \emph{The algebraic theory of modular systems}, Cambridge Mathematical Library. Cambridge  University Press, Cambridge, (1916). Reissued with an introduction by P.~Roberts, (1994). xxxii+112 pp. ISBN: 0-521-45562-6.

\bibitem[Mac3]{Mac2}
F.H.S.~Macaulay: \emph{Some properties of enumeration in the theory of modular systems},
Proc. London Math Soc. 26 (1927), 531-555.

\bibitem[MMS]{MMS}
P. Macias-Marques, C. McDaniel, and A. Seceleanu: \emph{Lorentzian polynomials, higher Hessians, and the Hodge-Riemann property for graded Artinian Gorenstein algebras}, https://arxiv.org/abs/2208.05653 .
\bibitem[MNZ1]{MNZ}
J. Migliore, U. Nagel, and F. Zanello:
\emph{Bounds and asymptotic minimal growth for Gorenstein Hilbert functions,} 
J. Algebra 321 (2009), no. 5, 1510-1521. 

\bibitem[MNZ2]{MNZ2}
J. Migliore, U. Nagel, and F. Zanello: \emph{Pure O-sequences: known results, applications, and open problems}, in Commutative algebra, 527-550, Springer, New York (2013).

\bibitem[MZ]{MZ}
J. Migliore and F. Zanello: \emph{
Unimodal Gorenstein $h$-vectors without the Stanley-Iarrobino property,} 
Comm. Algebra 46 (2018), no. 5, 2054-2062. 

\bibitem[MuNaYa]{MuNY}
S. Murai, T. Nagaoka, A. Yazawa: \emph{Strictness of the log-concavity of generating polynomials of matroids}, J. Combinatorial Theory A, 181 (2021), 105351.

\bibitem[Na]{Na}
M. Nagata: \emph{Local rings}, Interscience Tracts in Pure and Applied Mathematics, No. 13 Interscience Publishers, division of John Wiley \& Sons, New York-London (1962) xiii+234 pp.

   \bibitem[St1]{St}
R. Stanley: \emph{Hilbert functions of graded algebras}, Advances in Math. 28 (1978), 58-73.
  
  \bibitem[St2]{St2}
  R. Stanley: \emph{Log-concave and unimodal sequences in algebra, combinatorics, and geometry}, Graph theory and its applications: East and West (Jinan, 1986), 500-535, 
Ann. New York Acad. Sci., 576, New York Acad. Sci., New York, 1989.

\bibitem[Va]{Va}
G. Valla: \emph{Problems and results on Hilbert functions of graded algebras}, Six lectures on commutative algebra, 293-344, Mod. Birkh\"{a}user Class., Birkh\"{a}user Verlag, Basel, 2010.

  \bibitem[We]{We}
  A.J. Weiss: \emph{Some new non-unimodal level algebras}, Ph.D. thesis, Tufts University,
  arXiv:math.AC/0708.3354 (2007).
  
  \bibitem[Za1]{Za1}
  F. Zanello: \emph{The h-vector of a relatively compressed level algebra,}
Comm. Algebra 35 (2007), no. 4, 1087-1091.

  \bibitem[Za2]{Za2}
  F. Zanello: \emph{Log-concavity of level Hilbert functions and pure O-sequences}, (2022) arXiv:math.AC/2210.09447.
   \end{thebibliography}
\end{document}